\documentclass{article}
\usepackage{amsfonts}
\usepackage{amsmath}

\newtheorem{theorem}{Theorem}

\newtheorem{definition}[theorem]{Definition}
\newtheorem{example}[theorem]{Example}

\newtheorem{lemma}[theorem]{Lemma}

\newtheorem{proposition}[theorem]{Proposition}
\newtheorem{remark}[theorem]{Remark}

\newenvironment{proof}[1][Proof]{\noindent\textbf{#1.} }{\ \rule{0.5em}{0.5em}}
\input{tcilatex}

\begin{document}

\title{A Combinatorial Discussion on Finite Dimensional Leavitt Path Algebras%
}
\author{A. Ko\c{c}$^{(1)}$, S. Esin$^{(2)}$, \.{I}. G\"{u}lo\u{g}lu$^{(2)}$,
M. Kanuni$^{(3)}$ \\
%EndAName
$^{(1)}$ \.{I}stanbul K\"{u}lt\"{u}r University\\
Department of Mathematics and Computer Sciences\\
$^{(2)}$ Do\u{g}u\c{s} University\\
Department of Mathematics\\
$^{(3)}$ Bo\u{g}azi\c{c}i University\\
Department of Mathematics}
\maketitle

\begin{abstract}
Any finite dimensional semisimple algebra A over a field K is isomorphic to
a direct sum of finite dimensional full matrix rings over suitable division
rings. In this paper we will consider the special case where all division
rings are exactly the field K. All such finite dimensional semisimple
algebras arise as a finite dimensional Leavitt path algebra. For this
specific finite dimensional semisimple algebra $A$ over a field $K,$ we
define a uniquely detemined specific graph - which we name as a truncated
tree associated with $A$ - whose Leavitt path algebra is isomorphic to $A$.
We define an algebraic invariant $\kappa (A)$ for $A\ $and count the number
of isomorphism classes of Leavitt path algebras with $\kappa (A)=n.$

Moreover, we find the maximum and the minimum $K$-dimensions of the Leavitt
path algebras of possible trees with a given number of vertices and
determine the number of distinct Leavitt path algebras of a line graph with
a given number of vertices.
\end{abstract}


\section{Introduction}

By the well-known Wedderburn-Artin Theorem \ \cite{2}, any finite
dimensional semisimple algebra $A$ over a field $K$ is isomorphic to a
direct sum of finite dimensional full matrix rings over suitable division
rings. In this paper we will consider the special case where \ all division
rings are exactly the field $K.$ All such finite dimensional semisimple
algebras arise as a finite dimensional Leavitt path algebra as studied in 
\cite{1}. The Leavitt path algebras are introduced by Abrams and Aranda Pino
in 2005, \cite{3}. Many papers on Leavitt path algebras appeared in
literature since then. In the following discussion, we are particularly
interested in answering some combinatorial questions on the finite
dimensional Leavitt path algebras.

We start by recalling the definitions of a path algebra and a Leavitt path
algebra, see \cite{1}. A \textit{directed graph} $E=(E^{0},E^{1},r,s)$
consists of two countable sets $E^{0},E^{1}$ and functions $%
r,s:E^{1}\rightarrow E^{0}$. The elements $E^{0}$ and $E^{1}$ are called 
\textit{vertices} and \textit{edges}, respectively. For each $e\in E^{0},$ $%
s(e)$ is the source of $e$ and $r(e)$ is the range of $e.$ If $s(e)=v$ and $%
r(e)=w,$ then we say that $v$ emits $e$ and that $w$ receives $e.$ A vertex
which does not receive any edges is called a \textit{source,} and a vertex
which emits no edges is called a \textit{sink.} A graph is called \textit{%
row- finite} if $s^{-1}(v)$ is a finite set for each vertex $v$. For a
row-finite graph the edge set $E^{1}$ of $E~$is finite if its set of
vertices $E^{0}$ is finite. Thus, a row-finite graph is finite if $E^{0}$ is
a finite set.

A path in a graph $E$ is a sequence of edges $\mu =e_{1}\ldots e_{n}$ such
that $r(e_{i})=s(e_{i+1})$ for $i=1,\ldots ,n-1.$ In such a case, $s(\mu
):=s(e_{1})$ is the \textit{source }of $\mu $ and $r(\mu ):=r(e_{n})$ is the 
\textit{range} of $\mu $, and $n$ is the \textit{length }of $\mu ,$ i.e., $%
l(\mu )=n.$

If $s(\mu )=r(\mu )$ and $s(e_{i})\neq s(e_{j})$ for every $i\neq j$, then $%
\mu $ is called a \textit{cycle}. If $E$ does not contain any cycles, $E$ is
called \textit{acyclic}.

For $n\geq 2,$ define $E^{n}$ to be the set of paths of length $n,$ and $%
E^{\ast }=\bigcup\limits_{n\geq 0}E^{n}$ the set of all paths.

The path $K$-algebra over $E$ is defined as the free $K$-algebra $%
K[E^{0}\cup E^{1}]$ with the relations:

\begin{enumerate}
\item[(1)] $v_{i}v_{j}=\delta _{ij}v_{i}$ \ for every $v_{i},v_{j}\in E^{0}.$

\item[(2)] $e_{i}=e_{i}r(e_{i})=s(e_{i})e_{i}$ $\ $for every $e_{i}\in
E^{1}. $
\end{enumerate}

This algebra is denoted by $KE$. Given a graph $E,$ define the extended
graph of $E$ as the new graph $\widehat{E}=(E^{0},E^{1}\cup (E^{1})^{\ast
},r^{\prime },s^{\prime })$ where $(E^{1})^{\ast }=\{e_{i}^{\ast
}~|~e_{i}\in E^{1}\}$ and the functions $r^{\prime }$ and $s^{\prime }$ are
defined as 
\begin{equation*}
r^{\prime }|_{E^{1}}=r,~~~~s^{\prime }|_{E^{1}}=s,~~~~r^{\prime
}(e_{i}^{\ast })=s(e_{i})~~~~~~\text{and~~~~~}s^{\prime }(e_{i}^{\ast
})=r(e_{i}).
\end{equation*}%
The Leavitt path algebra of $E$ with coefficients in $K$ is defined as the
path algebra over the extended graph $\widehat{E},$ with relations:

\begin{enumerate}
\item[(CK1)] $e_{i}^{\ast }e_{j}=\delta _{ij}r(e_{j})$ \ for every $e_{j}\in
E^{1}$ and $e_{i}^{\ast }\in (E^{1})^{\ast }.$

\item[(CK2)] $v_{i}=\sum_{\{e_{j}\in
E^{1}~|~s(e_{j})=v_{i}\}}e_{j}e_{j}^{\ast }$ \ for every $v_{i}\in E^{0}$
which is not a sink.
\end{enumerate}

This algebra is denoted by $L_{K}(E)$. The conditions (CK1) and (CK2) are
called the Cuntz-Krieger relations. In particular condition (CK2) is the
Cuntz-Krieger relation at $v_{i}$. If $v_{i}$ is a sink, we do not have a
(CK2) relation at $v_{i}$. Note that the condition of row-finiteness is
needed in order to define the equation (CK2).

The main structure theorem in\ \cite{1} can be summarized as follows:

For any $v\in E^{0},$ we define $n(v)=\left\vert \left\{ \alpha \in E^{\ast
}~|~r(\alpha )=v\right\} \right\vert .$\newpage

\begin{proposition}
\label{prop1}:

\begin{enumerate}
\item The Leavitt path algebra $L_{K}(E)$ is a finite-dimensional $K$%
-algebra if and only if $E$ is a finite and acyclic graph.

\item If $A=\bigoplus\limits_{i=1}^{s}M_{n_{i}}(K)$ , then $A\cong L_{K}(E)$
for a graph $E$ having $s$ connected components each of which is an oriented
line graph with $n_{i}$ vertices, \linebreak $i=1,2,\cdots ,s.$

\item A finite dimensional $K$-algebra $A$ arises as a $L_{K}(E)$ for a
graph $E$ if and only if $A=\bigoplus\limits_{i=1}^{s}M_{n_{i}}(K).$

\item If $A=\bigoplus\limits_{i=1}^{s}M_{n_{i}}(K)$ and $A\cong L_{K}(E)$
for a finite, acyclic graph $E$, then the number of sinks of $E$ is equal to 
$s$, and each sink $v_{i}$ $(i=1,2,\cdots ,s)$ has $n(v_{i})=n_{i}$ with a
suitable indexing of the sinks.
\end{enumerate}
\end{proposition}

\section{Truncated Trees}

For a finite dimensional Leavitt path algebra $L_{K}(E)$ of a graph $E$, we
would like to construct a distinguished graph $F$ having the Leavitt path
algebra isomorphic to $L_{K}(E)$ as follows:

\begin{theorem}
\label{thm 2}Let $E$ be a finite, acyclic graph with no isolated points. Let
\linebreak $s=|S(E)|$ where $S(E)$ is the set of sinks of $E$ and $N=\max
\{n(v)~|~v\in S(E)\}$. Then there exists a unique (up to isomorphism) tree $%
F $ with exactly one source and $s+N-1$ vertices such that $L_{K}(E)\cong
L_{K}(F)$.
\end{theorem}

\begin{proof}
Let the sinks $v_{1},v_{2},\ldots ,v_{s}$ of $E$ be indexed \ such that 
\begin{equation*}
2\leq n(v_{1})\leq n(v_{2})\leq \ldots \leq n(v_{s})=N.
\end{equation*}%
Define a graph $F=(F^{0},F^{1},r,s)$ as follows:%
\begin{eqnarray*}
F^{0} &=&\{u_{1},u_{2},\ldots ,u_{N},w_{1},w_{2},\ldots w_{s-1}\} \\
F^{1} &=&\{e_{1},e_{2},\ldots ,e_{N-1},f_{1},f_{2},\ldots ,f_{s-1}\} \\
s(e_{i}) &=&u_{i}\text{ \ \ \ \ and \ \ \ }r(e_{i})=u_{i+1}\text{\ \ \ \ \ \
\ \ \ \ \ \ }i=1,\ldots ,N-1 \\
s(f_{i}) &=&u_{n(v_{i})-1}\text{ \ \ \ \ \ \ and \ \ \ \ }r(f_{i})=w_{i}%
\text{ \ \ \ \ \ \ }i=1,\ldots ,s-1.
\end{eqnarray*}%
\FRAME{dhF}{3.7317in}{1.1467in}{0pt}{}{}{Figure}{\special{language
"Scientific Word";type "GRAPHIC";maintain-aspect-ratio TRUE;display
"USEDEF";valid_file "T";width 3.7317in;height 1.1467in;depth
0pt;original-width 12.8961in;original-height 3.9271in;cropleft "0";croptop
"1";cropright "1";cropbottom "0";tempfilename
'LXQCLQ00.wmf';tempfile-properties "XPR";}}Clearly, $F$ is a directed tree
with unique source $u_{1}$ and $s+N-1$ vertices. $F$ has exactly $s$ sinks,
namely $u_{N},w_{1},w_{2},\ldots w_{s-1}$ with $n(u_{N})=N,$ $%
n(w_{i})=n(v_{i}),$ \ \ $i=1,\ldots ,s-1.$ Therefore, $L_{K}(E)\cong
L_{K}(F).$

For the uniqueness part, take a tree $T$ with exactly one source and
\linebreak $s+N-1$ vertices such that $L_{K}(E)\cong L_{K}(T)$. Since $%
N=\max \{n(v)~|~v\in S(E)\}$ which is equal to the square root of the
maximum of the $K$-dimensions of the minimal ideals of $L_{K}(E)$ and hence $%
L_{K}(T),$ there exists a sink $v$ in $T$ with $\left\vert \{\mu _{i}\in
T^{\ast }~|~r(\mu _{i})=v\}\right\vert =N.$ On the other hand, since $T$ is
a tree with a unique source and hence any vertex is connected to the unique
source by a uniquely determined path, we see that the unique path joining $v$
to the source must contain exactly $N$ \ vertices, say $a_{1},...,a_{N-1},v$
\ where $a_{1}$ is the unique source and the length of the path joining $%
a_{k}$ to $a_{1}$ being equal to $k-1$ for any \linebreak $k=1,2,...,N-1$.
As $L_{K}(E)=\bigoplus\limits_{i=1}^{s}M_{n_{i}}(K)$ with $s$ summands, the
remaining $s-1$ vertices must then all be sinks by Proposition \ref{prop1}
(4), say $b_{1},...,b_{s-1}.$ Since for any vertex $a$ different from the
unique source we have $n(a)>1$ we see that for each $i=1,\ldots ,s-1$ there
exists an edge $g_{i}$ with $r(g_{i})=b_{i}.$ Since $s(g_{i})$ is not a sink
we see that $s(g_{i})\in \{a_{1},a_{2},...,a_{N-1}\},$ more precisely $%
s(g_{i})=a_{n(b_{i})-1},$ $i=1,2,...,s-1.$ Thus \ $T$ is isomorphic to $F.$
\end{proof}

Observe that the $F$ constructed in Theorem \ref{thm 2} is the tree with one
source and smallest possible number of vertices $(s+N-1)$ having $L_{K}(F)$
isomorphic to $L_{K}(E).$ We call $F$ constructed in Theorem \ref{thm 2} as
the\textit{\ truncated tree associated with} $E.$

\begin{proposition}
With the above definition of $F$, there is no tree $T$ with \linebreak $%
|T^{0}|<|F^{0}|$ such that $L_{K}(T)\cong L_{K}(F).$
\end{proposition}

\begin{proof}
Notice that since $T$ is a tree, any vertex contributing to a sink
represents a unique path ending at that sink. Assume on the contrary there
exists a tree $T$ with $n$ vertices and $L_{K}(T)\cong
A=\bigoplus\limits_{i=1}^{s}M_{n_{i}}(K)$ such that $n<s+N-1.$ Since $N$ is
the maximum of $n_{i}$'$s$ there exists a sink with $N$ vertices
contributing. But in $T$ the number $n-s$ of vertices which are not sinks is
less than $N-1.$ Hence the maximum contribution to any sink can be at most $%
n-s+1$ which is strictly less than $N.$ This is the desired contradiction.
\end{proof}

However if we omit the tree assumption then it is possible to find a graph $%
G $ with smaller number of vertices having $L_{K}(G)$ isomorphic to $%
L_{K}(E) $ as the next example illustrates.

\begin{example}
\ \ \ \FRAME{dhF}{3.1678in}{0.7697in}{0pt}{}{}{Figure}{\special{language
"Scientific Word";type "GRAPHIC";maintain-aspect-ratio TRUE;display
"USEDEF";valid_file "T";width 3.1678in;height 0.7697in;depth
0pt;original-width 8.4475in;original-height 2.0314in;cropleft "0";croptop
"1";cropright "1";cropbottom "0";tempfilename
'LXL3CQ04.wmf';tempfile-properties "XPR";}}Both $L_{K}(G)\cong M_{3}(K)\cong
L_{K}(F)$ and $|G^{0}|=2$ where as $|F^{0}|=3$.
\end{example}

\bigskip

Given $F_{1}$, $F_{2}$ truncated trees associated with graphs $G_{1}$ and $%
G_{2}$ respectively, then $F_{1}\cong F_{2}$ iff $L_{K}(F_{1})\cong
L_{K}(F_{2})$ so there is a one-to-one correspondence between the Leavitt
path algebra and truncated trees.

For a given finite dimensional Leavitt path algebra $A=\bigoplus%
\limits_{i=1}^{s}M_{n_{i}}(K)$ with $2\leq n_{1}\leq n_{2}\leq \ldots \leq
n_{s}=N,$ the number $s$ is the number of minimal ideals of $A$ and $N^{2}$
is the maximum of the dimensions of these ideals. Therefore $\kappa
(A)=s+N-1 $ is a uniquely determined algebraic invariant of $A$. Given $%
m\geq 2$, the number of isomorphism classes of finite dimensional Leavitt
path algebras $A$ which do not have any ideals isomorphic to $K$ and $\kappa
(A)=m$ is equal to the number of distinct truncated trees with $m$ vertices
by the previous paragraph. The next proposition computes this number.

\begin{definition}
Define a function $d:E^{0}\rightarrow 
%TCIMACRO{\U{2115} }%
%BeginExpansion
\mathbb{N}
%EndExpansion
$ such that for any $u\in E^{0}$,%
\begin{equation*}
d(u)=\left\vert \{v~|~~n(v)\leq n(u)\}\right\vert .
\end{equation*}
\end{definition}

Observe that in a truncated tree, the restriction of the function $d$ on the
set of vertices which are not sinks is one to one.

\begin{proposition}
The number of distinct truncated trees with $n$ vertices is $2^{n-2}.$
\end{proposition}

\begin{proof}
For every truncated tree $E$ with $n$ vertices we assign an $n$-vector
\linebreak $\alpha (E)=(\alpha _{1},\alpha _{2},\cdots ,\alpha _{n})$ where $%
\alpha _{i}\in \{0,1\}$ as follows:

\begin{itemize}
\item $\alpha (E)$ contains exactly $N-1$ many $1$'s where $N-1$ is the
number of non-sinks of $E$.

\item To define that vector it is sufficient to know which component\ is $1.$

\item To each vertex $v$ which is not a sink, we assign a $1$ appearing in
the $d(v)$-$th$ component.

\item Remaining components are all zero.
\end{itemize}

Hence $\alpha (E)$ starts with $1$ and ends with $0$.

Given any $\{0,1\}$ sequence $\beta $ of length $n$ starting with $1$ and
ending with $0$, there exists clearly a unique truncated tree $E$ with $n$
vertices such that $\alpha (E)=\beta .$ Hence the number of distinct
truncated trees with $n$ vertices is equal to the number of all $\{0,1\}$%
-sequences of length $n$ in which the first and last components are constant
which is equal to $2^{n-2}.$
\end{proof}

\bigskip

For a tree $F$ with $n$ vertices the $K$-dimension of $L_{K}(F)$ is not
uniquely determined by the number of vertices only. However, we can compute
the maximum and the minimum $K$-dimensions of $L_{K}(F)$ where $F$ ranges
over all possible trees with $n$ vertices.

\begin{lemma}
\label{lemma7} The maximum $K$-dimension of $L_{K}(E)$ where $E$ ranges over
all possible trees with $n$ vertices and $s$ sinks is equal to $s(n-s+1)^{2} 
$.
\end{lemma}

\begin{proof}
Assume $E$ is a tree with $n$ vertices. Then $L_{K}(E)\cong
\bigoplus\limits_{i=1}^{s}M_{n_{i}}(K),$ by Proposition \ref{prop1} (3)
where $s$ is the number of sinks in $E$ and $n_{i}\leq n-s+1$ for all $%
i=1,\ldots s.$ Hence 
\begin{equation*}
\dim L_{K}(E)=\sum\limits_{i=1}^{s}n_{i}^{2}\leq s(n-s+1)^{2}.
\end{equation*}%
Notice that there exists a tree $E$ as sketched below\FRAME{dhF}{3.1254in}{%
0.921in}{0pt}{}{}{Figure}{\special{language "Scientific Word";type
"GRAPHIC";maintain-aspect-ratio TRUE;display "USEDEF";valid_file "T";width
3.1254in;height 0.921in;depth 0pt;original-width 8.3022in;original-height
2.4267in;cropleft "0";croptop "1";cropright "1";cropbottom "0";tempfilename
'LXL3CQ00.wmf';tempfile-properties "XPR";}}with $n$ vertices and\ $s$ sinks
such that $\dim L_{K}(E)=s(n-s+1)^{2}.$
\end{proof}

\begin{theorem}
The maximum $K$-dimension of $L_{K}(E)$ where $E$ ranges over all possible
trees with $n$ vertices is given by $f(n)$ where 
\begin{equation*}
f(n)=\left\{ 
\begin{array}{ccc}
\dfrac{n(2n+3)^{2}}{27} & if & n\equiv 0\text{ \ }(\func{mod}3) \\ 
\text{ \ \ \ } &  &  \\ 
\dfrac{1}{27}\left( n+2\right) \left( 2n+1\right) ^{2} & if & n\equiv 1\text{
\ }(\func{mod}3) \\ 
\text{ \ \ \ } &  &  \\ 
\dfrac{4}{27}(n+1)^{3} & if & n\equiv 2\text{ \ }(\func{mod}3)%
\end{array}%
\right.
\end{equation*}
\end{theorem}

\begin{proof}
Assume $E$ is a tree with $n$ vertices. Then $L_{K}(E)\cong
\bigoplus\limits_{i=1}^{s}M_{n_{i}}$ where $s$ is the number of sinks in $E$%
. Now, to find $\max \dim L_{K}(E)$ we need only to determine maximum value
of the function $f(s)=s(n-s+1)^{2}$ for \linebreak $s=1,2,\ldots ,n-1.$
Extending the domain of $f(s)$ to real numbers $1\leq s\leq n-1$ we get a
continuous function, hence we can find its maximum value.%
\begin{equation*}
f(s)=s(n-s+1)^{2}\Rightarrow \frac{d}{ds}\left( s(n-s+1)^{2}\right) =\left(
n-3s+1\right) \left( n-s+1\right)
\end{equation*}%
Then $s=\dfrac{n+1}{3}$ is the only critical point in the interval $\left[
1,n-1\right] $ and since $\dfrac{d^{2}f}{ds^{2}}(\dfrac{n+1}{3})<0,$ it is a
local maximum. In particular $f$ is increasing on $\left[ 1,\dfrac{n+1}{3}%
\right] $ and decreasing on $\left[ \dfrac{n+1}{3},n-1\right] $ . We have
three cases:

\textbf{Case 1: } $n\equiv 2$\ \ $(\func{mod}3).$ In this case $s=\dfrac{n+1%
}{3}$ is an integer and maximum $K$-dimension of $L_{K}(E)$ is $f\left( 
\dfrac{n+1}{3}\right) =\dfrac{4}{27}\left( n+1\right) ^{3}$ and we have $%
n_{i}=\dfrac{2(n+1)}{3},$ for each $i=1,2,\ldots ,s.\vspace{0.25in}$

\textbf{Case 2: }$n\equiv 0$ \ $(\func{mod}3).$ Then we have: $\frac{n}{3}%
=t<t+\dfrac{1}{3}=s<t+1$ and\newline
\begin{equation*}
f\left( \frac{n}{3}\right) =\frac{(2n+3)^{2}n}{27}=\alpha _{1}\text{ and }%
f\left( \frac{n}{3}+1\right) =\frac{4n^{2}(n+3)}{27}=\alpha _{2}.
\end{equation*}%
Note that, $\alpha _{1}>\alpha _{2}$. So $\alpha _{1}$ is maximum $K$%
-dimension of $L_{K}(E)$ and we have $n_{i}=\dfrac{2}{3}n+1,$ for each $%
i=1,2,\ldots ,s$.$\vspace{0.25in}$

\textbf{Case 3:} $n\equiv 1$ \ $(\func{mod}3).$ Then $\dfrac{n-1}{3}=$ $t<t+%
\dfrac{2}{3}=s<t+1$ and \newline
\begin{equation*}
f\left( \frac{n-1}{3}\right) =\frac{4}{27}\left( n+2\right) ^{2}\left(
n-1\right) =\beta _{1}\text{ }
\end{equation*}%
and 
\begin{equation*}
f\left( \frac{n+2}{3}\right) =\frac{1}{27}\left( 2n+1\right) ^{2}\left(
n+2\right) =\beta _{2}.
\end{equation*}%
In this case $\beta _{2}>\beta _{1}$ and so $\beta _{2}$ gives the maximum $%
K $-dimension of $L_{K}(E)$ and we have $n_{i}=\dfrac{2n+1}{3},$ for each $%
i=1,2,\ldots ,s$.
\end{proof}

\begin{theorem}
The minimum $K$-dimension of $L_{K}(E)$ where $E$ ranges over all possible
trees with $n$ vertices and $s$ sinks is equal to $r(q+2)^{2}+(s-r)(q+1)^{2}$%
, where $n-1=qs+r,~~0\leq r<s.$
\end{theorem}

\begin{proof}
We call a graph a \textit{bunch tree} if it is obtained by identifiying the
unique sources of the finitely many oriented finite line graphs.\FRAME{dhF}{%
2.3756in}{1.6864in}{0pt}{}{}{Figure}{\special{language "Scientific
Word";type "GRAPHIC";maintain-aspect-ratio TRUE;display "USEDEF";valid_file
"T";width 2.3756in;height 1.6864in;depth 0pt;original-width
6.4584in;original-height 4.5731in;cropleft "0";croptop "1";cropright
"1";cropbottom "0";tempfilename 'LXL3CQ01.wmf';tempfile-properties "XPR";}}%
Let $\mathcal{E}(n,s)$ be the set of all bunch trees with $n$ vertices and $%
s $ sinks.

Every element of $\mathcal{E}(n,s)$ can be uniquely represented by an $s$%
-tuple $(t_{1},t_{2},...,t_{s})$ where each $t_{i}$ is the number of
vertices contributing only to the $i$-th$^{\text{ }}$sink with $1\leq
t_{1}\leq t_{2}\leq ...\leq t_{s}$ and $t_{1}+t_{2}+...+t_{s}=n-1.$

Let $E\in \mathcal{E}(n,s)$ with $t_{s}-t_{1}\leq 1$. This $E$ is
represented by the $s$-tuple $(q,\ldots ,q,\underset{}{q+1,\ldots ,q}+1)$
where $n-1=sq+r$, $0\leq r<s.$

Now we claim that the dimension of $E$ is the minimum of the set 
\begin{equation*}
\left\{ \dim L_{K}(F):F\text{ tree with }s\text{ sinks and }n\text{ vertices}%
\right\} .
\end{equation*}%
If we represent $U\in \mathcal{E}(n,s)$ by the $s$-tuple $%
(u_{1},u_{2},...,u_{s})$ then $E\neq U$ implies that $u_{s}-u_{1}\geq 2.$

Consider the $s$-tuple $(t_{1},t_{2},...,t_{s})$ where $%
(t_{1},t_{2},...,t_{s})$ is obtained from \linebreak $%
(u_{1}+1,u_{2},...,u_{s-1},u_{s}-1)$ by reordering the components in
increasing order.

In this case the dimension $d_{U}$ of $U$ is%
\begin{equation*}
d_{U}=(u_{1}+1)^{2}+\ldots +(u_{s}+1)^{2}.
\end{equation*}%
Similarly, the dimension $d_{T}$ of the bunch graph $T$ represented by the $%
s $-tuple $(t_{1},t_{2},...,t_{s}),$ is%
\begin{equation*}
d_{T}=(t_{1}+1)^{2}+\ldots +(t_{s}+1)^{2}=(u_{1}+2)^{2}+\ldots
+u_{s-1}^{2}+u_{s}{}^{2}.
\end{equation*}%
Hence 
\begin{equation*}
d_{U}-d_{T}=2(u_{s}-u_{1})-2>0.
\end{equation*}%
Repeating this process sufficiently many times we see that the process has
to end at the exceptional bunch tree $E$ showing that its dimension is the
smallest among the dimensions of all elements of $\mathcal{E}(n,s)$.

Now let $F$ be an arbitrary tree with $n$ vertices and $s$ sinks. As above
we assign to $F$ the $s$-tuple $(n_{1},n_{2},...,n_{s})$ with $%
n_{i}=n(v_{i})-1$ where the sinks $v_{i},~i=1,2,\ldots ,s$ are indexed in
such a way that $n_{i}\leq n_{i+1},~i=1,\ldots ,s-1.$ Observe that $%
n_{1}+n_{2}+\cdots +n_{s}\geq n-1$. Let $\beta
=\dsum\limits_{i=1}^{s}n_{i}-(n-1).$ Since $s\leq n-1$, $\beta \leq
\dsum\limits_{i=1}^{s}(n_{i}-1).$ Either $n_{1}-1\geq \beta $ or there
exists a unique $k\in \left\{ 2,\ldots ,s\right\} $ such that $%
\dsum\limits_{i=1}^{k-1}(n_{i}-1)<\beta \leq \dsum\limits_{i=1}^{k}(n_{i}-1)$%
. If $n_{1}-1\geq \beta ,$ then let 
\begin{equation*}
m_{i}=\left\{ 
\begin{array}{ccc}
n_{1}-\beta & , & i=1 \\ 
n_{i} & , & i>1%
\end{array}%
\right. .
\end{equation*}%
Otherwise, let 
\begin{equation*}
m_{i}=\left\{ 
\begin{array}{ccc}
1 & , & i\leq k-1 \\ 
n_{k}-\left( \beta -\sum\limits_{i=1}^{k-1}(n_{i}-1)\right) & , & i=k \\ 
n_{i} & , & i\geq k+1%
\end{array}%
\right. .
\end{equation*}%
In both cases, the $s$-tuple $(m_{1},m_{2},\ldots ,m_{s})$ that satisfies $%
1\leq m_{i}\leq n_{i}$, \linebreak $m_{1}\leq m_{2}\leq \cdots \leq m_{s}$
and $m_{1}+m_{2}+\cdots +m_{s}=n-1$ is obtained. So, there exists a bunch
tree $M$ namely the one corresponding uniquely to $(m_{1},m_{2},\ldots
,m_{s})$ which has dimension $d_{M}\leq d_{F}.$ This implies that $d_{F}\geq
d_{E}.$

Hence the result follows.
\end{proof}

\begin{lemma}
The minimum $K$-dimension of $L_{K}(E)$ where $E$ ranges over all possible
trees with $n$ vertices occurs when the number of sinks is $n-1$ and is
equal to $4(n-1)$.
\end{lemma}

\begin{proof}
By the previous theorem we see that 
\begin{equation*}
\dim L_{K}(E)\geq r(q+2)^{2}+(s-r)(q+1)^{2}
\end{equation*}%
where $n-1=qs+r,~~0\leq r<s.$ We have 
\begin{equation*}
r(q+2)^{2}+(s-r)(q+1)^{2}=(n-1)(q+2)+qr+r+s.
\end{equation*}%
Thus 
\begin{equation*}
(n-1)(q+2)+qr+r+s-4(n-1)=(n-1)(q-2)+qr+r+s\geq 0\text{ \ }if\text{ \ }q\geq
2.
\end{equation*}%
If $q=1,$then $-(n-1)+2r+s=-(n-1)+r+(n-1)=r\geq 0.$ Hence $\dim L_{K}(E)\geq
4(n-1).$

Notice that there exists a truncated tree $E$ with $n$ vertices and
\linebreak $\dim L_{K}(E)=4(n-1)$ as sketched below : \FRAME{dhF}{1.3828in}{%
1.0136in}{0pt}{}{}{Figure}{\special{language "Scientific Word";type
"GRAPHIC";maintain-aspect-ratio TRUE;display "USEDEF";valid_file "T";width
1.3828in;height 1.0136in;depth 0pt;original-width 3.4065in;original-height
2.4898in;cropleft "0";croptop "1";cropright "1";cropbottom "0";tempfilename
'LXL3CQ02.wmf';tempfile-properties "XPR";}}
\end{proof}

\section{Line Graphs}

The \textit{total-degree} of the vertex $v$ is the number of edges that
either have $v$ as its source or as its range, that is, $tot\deg
(v)=\left\vert s^{-1}(v)\cup r^{-1}(v)\right\vert .$ A finite graph $E$ is a 
\textit{line graph} if it is connected, acyclic and $totdeg(v)\leq 2$ for
every $v\in E^{0}.$

\begin{remark}
In \cite{1}, the proposition 5.7 shows that a semisimple finite dimensional
algebra $A=\bigoplus\limits_{i=1}^{s}M_{n_{i}}(K)$ over the field $K$ can be
described as a Leavitt path algebra $L(E)$ defined \ by a line graph $E,$ if
and only if $A$ has no ideals of $K-$dimension $1$ and the number of minimal
ideals of $A$ of $K$ dimension $2^{2}$ is at most $2.$ On the other hand, if 
$A\cong L(E)$ for some $n$ line graph $E$ then $n-1=\sum%
\limits_{i=1}^{s}(n_{i}-1),$ that is, $n$ is an algebraic invariant of $A.$
\end{remark}

Therefore the following proposition answers a reasonable question.

\begin{proposition}
The number $A_{n}$ of isomorphism classes of Leavitt path algebras defined
by line graphs having exactly $n$ vertices is%
\begin{equation*}
A_{n}=P(n-1)-P(n-4)
\end{equation*}%
where $P(m)$ is the number of partitions of the natural number $m.$
\end{proposition}

\begin{proof}
Any $n$-line graph has $n-1$ edges. In a line graph, for any edge $e$ there
exists a unique sink $v$ so that there exists a path from $s(e)$ to $v.$ In
this case we say that $e$ is directed towards $v$. The number of edges
directed towards $v$ is clearly equal to $n(v)-1.$ Let $E$ and $F$ be two $n$%
-line graphs. $L_{K}(E)\cong L_{K}(F)$ if and only if there exists a
bijection $\phi :S(E)\rightarrow S(F)$ such that for each $v$ in $S(E),$ we
have $n(v)=n(\phi (v)).$ Therefore the number of isomorphism classes of
Leavitt path algebras determined by $n$-line graphs is the number of
partitions of $n-1$ edges in which the number of parts having exactly one
edge is at most two. Since the number of partitions of $k$ objects having at
least three parts each of which containing exactly one element is $P(k-3)$,
we get the result $A_{n}=P(n-1)-P(n-4).$
\end{proof}

\end{document}